\documentclass{amsart}%
\usepackage{amsfonts}
\usepackage{amsmath}
\usepackage{amssymb}
\usepackage{graphicx}%
\providecommand{\U}[1]{\protect\rule{.1in}{.1in}}
\newtheorem{theorem}{Theorem}
\theoremstyle{plain}

\newtheorem{corollary}[theorem]{Corollary}

\newtheorem{lemma}[theorem]{Lemma}

\newtheorem{proposition}[theorem]{Proposition}

\numberwithin{equation}{section}
\numberwithin{theorem}{section}
\begin{document}
\title[Calculation of joint spectral radii]{Application of topological radicals to calculation of joint spectral radii}
\author{Victor S. Shulman}
\address{Department of Mathematics, Vologda State University, 15 lenina Str., Vologda
160000, Russian Federation}
\email{shulman\_v@yahoo.com}
\author{Yuri V. Turovskii}
\address{Institute of Matehematics and Mechanics, National Academy of Sciences of
Azerbaijan, 9 F. Agaev Str., Baku AZ1141, Azerbaijan}
\email{yuri.turovskii@gmail.com}
\thanks{The support received from INTAS project No 06-1000017-8609 is gratefully
acknowledged by the second author}
\date{}
\subjclass[2000]{Primary 47D03; Secondary 46H05} \keywords{Joint spectral radius, Berger-Wang radius,
Berger-Wang formula, generalized Berger-Wang formula, invariant subspace} \dedicatory{To the memory of our
fathers  Semen Moiseevich Shulman and Vladimir Vasilyevich Turovskii, the Soviet Army officers who fought
against fascism in the World War II}
\begin{abstract}
It is shown that the joint spectral radius $\rho(M)$ of a precompact family
$M$ of operators on a Banach space $X$ is equal to the maximum of two numbers:
the joint spectral radius $\rho_{e}(M)$ of the image of $M$ in the Calkin
algebra and the Berger-Wang radius $r(M)$ defined by the formula
\[
r(M)=\underset{n\rightarrow\infty}{\limsup}\left(  \sup\left\{  \rho(a):a\in
M^{n}\right\}  ^{1/n}\right)  .
\]
Some more general Banach-algebraic results of this kind are also proved. The
proofs are based on the study of special radicals on the class of Banach algebras.

\end{abstract}
\maketitle

\section{Introduction and preliminaries}

In 1960 Rota and Strang \cite{RS} introduced the notion of spectral radius for
sets of operators or, more generally, of elements of a Banach algebra. Namely,
if $M$ is a bounded subset of a Banach algebra $A$, the \textit{joint spectral
radius} $\rho(M)$ is defined by
\begin{equation}
\rho(M)=\underset{n\rightarrow\infty}{\lim}\left\Vert M^{n}\right\Vert
^{1/n}=\inf_{n}\left\Vert M^{n}\right\Vert ^{1/n},\label{nach}%
\end{equation}
where the \textit{norm} of a bounded set is the supremum of the norms of its
elements, and the products of sets are defined elementwise:
\[
M_{1}M_{2}=\{ab:a\in M_{1},b\in M_{2}\}.
\]

The notion turned out to be useful in various branches of mathematics:
wavelets, evolution dynamics, difference equations and the operator theory
itself. In \cite{Sh84} the joint spectral radius was applied to show
\textit{the existence of hyperinvariant subspace for every operator algebra
whose Jacobson radical contains non-zero compact operators.} This stimulates
the interest in convenient ways for the calculation of $\rho(M)$.

The map $M\longmapsto\rho(M)$ has many convenient algebraic and analytic
properties, in particular it is subharmonic \cite[Theorem 3.5]{ShT2000} (for
finite $M$, see also \cite[Theorem 3.8]{T85}). The latter means that if
$M=M(\lambda)$ analytically depends on a complex parameter $\lambda$ under
natural conditions then $\lambda\longmapsto\rho(M(\lambda))$ is a subharmonic
function. The following property \cite[Corollary 2.10]{ShT2000} is quite
important (see also a stronger result in \cite[Proposition 3.5]{T85}).

\begin{lemma}
\label{01} If $\rho(M)=0$ then all elements in the subalgebra generated by $M$
are quasinilpotent.
\end{lemma}

A very useful formula for the joint spectral radius of a bounded set of matrices was found in 1992 by M. A.
Berger and Y. Wang \cite{BW}. To formulate it we introduce, following to \cite{BW}, another spectral
characteristics of a bounded subset $M$ of a Banach algebra:
\[
r(M)=\underset{n\rightarrow\infty}{\lim\sup}\left(  \sup\left\{  \rho(a):a\in
M^{n}\right\}  \right)  ^{1/n}.
\]
It is clear that
\[
r(M)\leq\rho(M).
\]
It was proved in \cite{BW} that
\begin{equation}
r(M)=\rho(M)\label{bw}%
\end{equation}
for any bounded set of matrices. This equality (the\textit{ Berger-Wang
formula}) was extended in \cite{ShT2000} to arbitrary precompact sets of
compact operators on Banach spaces.

It should be noted that both restrictions of compactness are essential. It is
proved in \cite{Gu} that there are two non-compact operators $a,b$ such that
\[
r(\{a,b\})=0\neq\rho(\{a,b\}).
\]
There are also bounded families of compact operators for which the equality
(\ref{bw}) fails (see for instance \cite{PW}).

The Berger-Wang formula for compact operators was applied in \cite{ShT2000} to
the study of operator semigroups and Lie algebras. As a simplest example of
its application, note that it implies immediately the existence of a
nontrivial closed invariant subspace for a semigroup of \textit{Volterra}
(i.e. compact quasinilpotent) operators, established in \cite{Tur98}.

\begin{quotation}
{\small Indeed, if }$G$ {\small  is a semigroup of Volterra operators then for
each finite subset }$M$ {\small  of }$G${\small , every power }$M^{n}$
{\small  consists of quasinilpotent operators, whence }$r(M)=0${\small . By
(\ref{bw}), }$\rho(M)=0$, {\small  whence the linear span of }$M$ {\small
consists of Volterra operators by Lemma \ref{01}. Thus the linear span of }$G$
{\small  is an algebra of Volterra operators. By the Lomonosov Theorem
\cite{Lom}, it has a nontrivial closed invariant subspace.}
\end{quotation}

Our main aim here is to prove that, for arbitrary precompact set $M$ of
operators on $\mathcal{X}$,%
\begin{equation}
\rho(M)=\max\{r(M),\rho_{e}(M)\},\label{gbwf}%
\end{equation}
where $\rho_{e}(M)$ \ $=\rho(\pi(M))$, the joint spectral radius of the image
of $M$ in the Calkin algebra $\mathcal{B}\left(  \mathcal{X}\right)
/\mathcal{K}\left(  \mathcal{X}\right)  $ under the canonical homomorphism
$\pi$.

We call (\ref{gbwf}) the \textit{generalized Berger-Wang formula}. This
formula not only extends (\ref{bw}) to arbitrary operators, but also gives
many additional possibilities for applications. Note for example that it
implies immediately that (\ref{bw}) \textit{holds for precompact families of
operators of the form }$\lambda1+K$\textit{, where }$\lambda\in\mathbb{C}%
$\textit{ and }$K$\textit{ is a compact operator.}

\begin{quotation}
{\small  Indeed, in this case }$\pi(M)$ {\small  consists of scalar multiples
of the unit in the Calkin algebra }$\mathcal{B}\left(  \mathcal{X}\right)
/\mathcal{K}\left(  \mathcal{X}\right)  ${\small , whence }%
\[
\rho_{e}(M)=r({\pi}(M))\leq r(M),
\]
{\small and (\ref{gbwf}) shows that}%
\[
\rho(M)\leq r(M).
\]

\end{quotation}

In its turn the equality (\ref{bw}) for `scalar plus compact' operators was a
main ingredient of the proof (in \cite{ShT2000}) that \textit{any Lie algebra
of compact quasinilpotent operators has a nontrivial closed invariant
subspace.}

\begin{quotation}
{\small Let us recall the proof of this result. Let }$L$ {\small be a Lie
subalgebra in }$\mathcal{B}(\mathcal{X})${\small , that is a subspace of
}$\mathcal{B}(\mathcal{X})$ {\small such that }%
\[
TS-ST\in L
\]
{\small for all }$T,S\in L${\small . If }$L$ {\small consists of
quasinilpotent operators then }%
\[
G=\{e^{T}:T\in L\}
\]
{\small is a group (Wojty\'{n}ski's version of the Baker-Campbell-Hausdorff
Theorem \cite{WW}), and all operators in }$G$ {\small have spectrum }%
$\{1\}${\small . It follows that }%
\[
r(M)=1
\]
{\small for each finite set }$M\subset G${\small . If }$L$ {\small consists of
compact operators then }$G$ {\small consists of operators of the form }$1+K$
{\small with compact }$K${\small . By the above, }%
\[
\rho(M)=1,
\]
{\small for each finite subset\ }$M\subset G${\small . Choose }$T\in M$
{\small and define a function }$f(z)$ {\small on }$\mathbb{C}$ {\small by }%
\[
f(z)=\rho(M(e^{zT}-1)/z).
\]
{\small It is subharmonic and tends to zero on infinity because }%
\[
\rho(M(e^{zT}-1))\leq2
\]
{\small (use the fact that the joint spectral radius of a bounded set is not
changed if pass to closed convex hull of this set). Hence we have}%
\[
f(z)=0,
\]
{\small for every }$z\in\mathbb{C}${\small . In particular, }$f(0)=0$
{\small and }%
\[
\rho(TM)=0.
\]
{\small Now, by Lemma \ref{01}, }%
\[
\rho(TS)=0
\]
{\small for any }$S\in A${\small , where }$A$ {\small is the linear span of
}$G${\small . Since }$A$ {\small is an algebra, it has an invariant subspace
by Lomonosov's Lemma. But it is straightforward that an invariant subspace for
}$G$ {\small is invariant for }$L${\small .}
\end{quotation}

Note that the definition of $\rho_{e}(M)$ can be rewritten as follows
\[
\rho_{e}(M)=\underset{n\rightarrow\infty}{\lim\sup}\left(  \sup\left\{
\left\Vert T\right\Vert _{e}:T\in M^{n}\right\}  \right)  ^{1/n},
\]
where $||T||_{e}=||\pi(T)||$ is the essential norm of $T$ ($T\longmapsto
||T||_{e}$ is a seminorm on $\mathcal{B}\left(  \mathcal{X}\right)  $). It is
somewhat more convenient is to use%
\[
\rho_{\chi}(M)=\underset{n\rightarrow\infty}{\lim\sup}\left(  \sup\left\{
\left\Vert T\right\Vert _{\chi}:T\in M^{n}\right\}  \right)  ^{1/n}.
\]
Here $\left\Vert T\right\Vert _{\chi\text{ }}$ is the Hausdorff measure of
non-compactness for $T\mathcal{X}_{1}$, where
\[
\mathcal{X}_{1}=\left\{  x\in\mathcal{X}:\left\Vert x\right\Vert
\leq1\right\}  .
\]
In other words, $\left\Vert T\right\Vert _{\chi\text{ }}$is the infimum of all
such $\varepsilon>0$ that $T\mathcal{X}_{1}$ can be covered by a finite number
of balls of radius $\varepsilon$. The advantage of the work with this seminorm
is that
\[
\left\Vert T|\mathcal{Y}\right\Vert _{\chi}\leq\left\Vert T\right\Vert
_{\chi\text{ }}%
\]
and
\begin{equation}
\left\Vert T|\left(  \mathcal{X}/\mathcal{Y}\right)  \right\Vert _{\chi}%
\leq\left\Vert T\right\Vert _{\chi},\label{restQuot}%
\end{equation}
where $T|\mathcal{Y}$ is the restriction of $T$ to an invariant subspace
$\mathcal{Y}$ and $\ T|\left(  \mathcal{X}/\mathcal{Y}\right)  $ is the
operator induced by $T$ on the quotient space $\mathcal{X}/\mathcal{Y}$.

So we will prove that%
\begin{equation}
\rho\left(  M\right)  =\max\left\{  r\left(  M\right)  ,\rho_{\chi}\left(
M\right)  \right\}  \label{ggbwf}%
\end{equation}
for any precompact set $M$ of operators. Since
\[
\rho_{\chi}(M)\leq\rho_{e}(M)\leq\rho(M),
\]
this is a more strong result than (\ref{gbwf}).

The formula (\ref{ggbwf}) was proved in \cite{ShT2002} for operators on
reflexive spaces and, more generally, for weakly compact operators (see a
stronger result in \cite[Theorem 4.4]{ShT2002}). In the present work we are
able to remove these restrictions by applying the theory of  topological
radicals initiated by P. G. Dixon \cite{Dix} and further developed in
\cite{rad1}. We heavily use the properties of the radical $\mathcal{R}%
_{\mathrm{cq}}$ defined in \cite{rad1} in terms of the joint spectral radius,
and introduce a new topological radical, $\mathcal{R}_{\mathrm{hc}}$, related
to the compactness conditions.

We combine Banach algebraic and operator theoretic approaches. The first one
makes the subject more flexible and allows us to approximate to a needed
result step by step, `removing' more and more large ideals and passing to the
quotients (this process is simplified by means of the theory of radicals). The
second one enjoys the possibility to pass to the restrictions of operators to
invariant subspaces and to quotient spaces (which is especially important for
calculation of spectral radii). Our main results also have both Banach
algebraic and operator theoretic nature. We prove first a formula which
expresses the joint spectral radius of a family of elements of a Banach
algebra via the Hausdorff radius of a related family of multiplication
operators (we call this formula the \textit{mixed} GBWF). It is used to deduce
(\ref{ggbwf}) (the \textit{operator} GBWF). Then using (\ref{ggbwf}) we obtain
an extension of (\ref{gbwf}) to general Banach algebras (the \textit{Banach
algebraic} GBWF). The latter formula is of the same form as (\ref{gbwf}), but
the ideal $\mathcal{K}(\mathcal{X})$ is changed by the hypocompact radical
$\mathcal{R}_{\mathrm{hc}}(A)$ of a Banach algebra $A$. It should be noted
that, for $A=\mathcal{B}(\mathcal{X})$,
\[
\mathcal{R}_{\mathrm{hc}}(A)\supset\mathcal{K}(\mathcal{X})
\]
and the inclusion is proper for some Banach spaces. Hence the Banach-algebraic
formula not only extends (\ref{gbwf}) but also strengthens it.

We denote by $A^{1}$ the Banach algebra obtained by adjoining a unit to $A$
(if $A$ is unital then we put $A^{1}=A$). The term \textit{ ideal} always
means a closed two-sided ideal, and the term \textit{operator }always means a
bounded linear operator. If $J$ is an ideal of $A$ then by $q_{J}$ we denote
the quotient map from $A$ to $A/J$. The image of a set $M\subset A$ under
$q_{J}$ is denoted by either $q_{J}(M)$ or, simply, $M/J$. All spaces are
assumed to be over the field $\mathbb{C}$. If $M$ is a subset in a Banach
space then $\operatorname*{span}M$ denotes its closed linear span.

\section{Auxiliary results}

A natural way from the Banach algebraic setting to the operator one is to use
multiplication operators. As usual, for an element $a$ of a Banach algebra
$A$, we denote by $\mathrm{L}_{a}$ and $\mathrm{R}_{a}$ the operators of the
left and right multiplications by $a$ on $A$ defined by
\[
\mathrm{L}_{a}x=ax,\quad\mathrm{R}_{a}x=xa.
\]
Furthermore, for $M,N\subset A$, let
\[
\mathrm{L}_{M}=\{\mathrm{L}_{a}:a\in M\}
\]
and, similarly,
\[
\mathrm{R}_{N}=\{\mathrm{R}_{a}:a\in N\}.
\]
Multiplying these sets of operators, we define $\mathrm{L}_{M}\mathrm{R}_{N}$
by
\[
\mathrm{L}_{M}\mathrm{R}_{N}=\{\mathrm{L}_{a}\mathrm{R}_{b}:a\in M,b\in N\}.
\]

The joint spectral properties of $M$ are reflected in the properties of the
family $\mathrm{L}_{M}\mathrm{R}_{M}$.

\begin{lemma}
\label{pass}Let $M$ be a bounded set $M$ in a Banach algebra $A$. Then

\begin{itemize}
\item[(i)] $\rho(\mathrm{L}_{M}\mathrm{R}_{M})=\rho(M)^{2}$.

\item[(ii)] $r(\mathrm{L}_{M}\mathrm{R}_{M})=r(M)^{2}$.
\end{itemize}
\end{lemma}

\begin{proof}
(i) Note first of all that
\begin{equation}
\rho(M^{k})=\rho(M)^{k}\label{pow1}%
\end{equation}
for every bounded subset $M$ of $A$ and integer $k$. This follows from the
facts that
\[
M^{km}=\left(  M^{k}\right)  ^{m}%
\]
and that in (\ref{nach}) one can pass to a subsequence under $n=mk$.

It is clear that
\[
\Vert\mathrm{L}_{M}\mathrm{R}_{N}\Vert\leq\Vert\mathrm{L}_{M}\Vert
\Vert\mathrm{R}_{N}\Vert\leq\Vert M\Vert\Vert N\Vert
\]
for bounded subsets $M$ and $N$ of $A$. Also the formula
\[
(\mathrm{L}_{M}\mathrm{R}_{N})^{n}=\mathrm{L}_{M^{n}}\mathrm{R}_{N^{n}}%
\]
is evident. It follows from this that
\[
\rho(\mathrm{L}_{M}\mathrm{R}_{M})\leq\rho(M)^{2}.
\]

To show the converse, we note that
\[
\Vert M^{3}\Vert=\Vert\mathrm{L}_{M}\mathrm{R}_{M}(M)\Vert\leq\Vert
\mathrm{L}_{M}\mathrm{R}_{M}\Vert\Vert M\Vert.
\]
Changing $M$ by $M^{n}$, taking $n$-roots and limits, one obtains that%
\[
\rho(M)^{3}=\rho(M^{3})\leq\rho(\mathrm{L}_{M}\mathrm{R}_{M})\rho(M),
\]
whence
\[
\rho(M)^{2}\leq\rho(\mathrm{L}_{M}\mathrm{R}_{M}).
\]

(ii) Since $\mathrm{L}_{M}$ commutes with $\mathrm{R}_{M}$, we have that
\[
\rho(\mathrm{L}_{a}\mathrm{R}_{b})\leq\rho(\mathrm{L}_{a})\rho(\mathrm{R}%
_{b})\leq\rho(a)\rho(b)
\]
for every $a,b\in M^{n}$. This shows that
\[
r(\mathrm{L}_{M}\mathrm{R}_{M})\leq r(M)^{2}.
\]

On the other hand, for any $a\in M^{k}$, we obtain that
\begin{align*}
\rho(a)^{2} &  =\lim_{n\rightarrow\infty}\Vert a^{2n+1}\Vert^{1/n}%
=\lim_{n\rightarrow\infty}\Vert(\mathrm{L}_{a}\mathrm{R}_{a})^{n}%
(a)\Vert^{1/n}\leq\rho(\mathrm{L}_{a}\mathrm{R}_{a})\\
&  \leq\sup_{x,y\in M^{k}}\rho(\mathrm{L}_{x}\mathrm{R}_{y}).
\end{align*}
Taking supremum over all choices of $a_{i}$, one gets that
\[
\sup_{a\in M^{k}}\rho(a)^{2}\leq\sup_{x,y\in M^{k}}\rho(\mathrm{L}%
_{x}\mathrm{R}_{y}).
\]
It remains to take $k$-roots and pass to the upper limit to obtain that
\[
r(M)^{2}\leq r(\mathrm{L}_{M}\mathrm{R}_{M}).
\]

\end{proof}

The following result was proved in \cite[Corollary 6.5]{ShT2000}.

\begin{lemma}
\label{ineq} $\Vert\mathrm{L}_{M}\mathrm{R}_{M}\Vert_{\chi}\leq16\Vert
M\Vert_{\chi}\Vert M\Vert$ for each precompact set of operators.
\end{lemma}

Let us define $\rho^{\chi}(M)$ by
\[
\rho^{\chi}(M)=\rho_{\chi}(\mathrm{L}_{M}\mathrm{R}_{M})^{1/2},
\]
for a bounded set $M$ of elements of a Banach algebra.

It is easy to check, using (\ref{restQuot}), that
\[
\rho^{\chi}(M/J)\leq\rho^{\chi}(M),
\]
for any closed ideal $J$ (see \cite{ShT2002}).

An element $a\in A$ is called \textit{compact} if the operator
\[
\mathrm{W}_{a}=\mathrm{L}_{a}\mathrm{R}_{a}%
\]
is compact on $A$. More generally, a set $M\subset A$ \textit{consists of
mutually compact elements} if all operators in $\mathrm{L}_{M}\mathrm{R}_{M}$
are compact. We will need the following extension of the main result of
\cite{Tur98}.

\begin{lemma}
\label{semig} If $G$ is a semigroup of quasinilpotent mutually compact
elements of a Banach algebra $A$ then $\operatorname*{span}G$ consists of
quasinilpotent elements.
\end{lemma}

\begin{proof}
Note that $\mathrm{L}_{G}\mathrm{R}_{G}$ is a semigroup of compact
quasinilpotent operators on the Banach space $A$. By \cite{Tur98},
$\operatorname*{span}\mathrm{L}_{G}\mathrm{R}_{G}$ consists of quasinilpotent
elements. Note that $\mathrm{L}_{b}\mathrm{R}_{c}\in\operatorname*{span}%
\mathrm{L}_{G}\mathrm{R}_{G}$ for every $b,c\in\operatorname*{span}G$. Hence
$\mathrm{L}_{\operatorname*{span}G}\mathrm{R}_{\operatorname*{span}G}$
consists of quasinilpotents and the same is true for $\operatorname*{span}G$.
\end{proof}

\section{Hereditary topological radicals}

We deal here with a class of topological radicals that have especially
convenient properties. A \textit{hereditary topological radical} on the class
of all Banach algebras is a map $P$ which assigns to each Banach algebra $A$
its ideal $P(A)$ and satisfies the following conditions:

\begin{itemize}
\item[$\left(  \mathrm{R1}\right)  $] $P(A/P(A))=0$.

\item[$\left(  \mathrm{R2}\right)  $] $P(J)=J\cap P(A)$ for any ideal $J$ of
$A$.

\item[$\left(  \mathrm{R3}\right)  $] $f(P(A))\subset P(B)$ for continuous
surjective homomorphism $f:A\rightarrow B$.
\end{itemize}

It follows immediately from (R2) that
\[
P(P(A))=P(A).
\]
An algebra is called $P$\textit{-radical} if
\[
A=P(A).
\]
It can be proved (see \cite{rad1}) that ideals and quotients of $P$-radical
algebras are $P$-radical and that the class of all $P$-radical algebras is
stable with respect to extensions (if $J$ and $A/J$ are $P$-radical then $A$
is $P$-radical). We will need a more general result, also proved in
\cite{rad1}.

Let us call a transfinite increasing sequence $(J_{\alpha})_{\alpha\leq \gamma}$ of ideals in a Banach algebra
$A$ \textit{a transfinite increasing chain of ideals} if $J_{\beta}=\overline{\cup_{\alpha<\beta}J_{\alpha}}$
for any limit ordinal $\beta\leq\gamma$.

\begin{lemma}
\label{transext} Let $P$ be a hereditary topological radical. If in a
transfinite increasing chain of ideals $(J_{\alpha})_{\alpha\leq\gamma}$ of a
Banach algebra $A$ the first element $J_{1}$ and all quotients $J_{\alpha
+1}/J_{\alpha}$ are $P$-radical then $J_{\gamma}$ is $P$-radical.
\end{lemma}

The most popular example of a hereditary topological radical is the Jacobson
radical $\operatorname*{Rad}$. We need some other examples.

\subsection{The radical $\mathcal{R}_{cq}$}

All definitions and results of this subsection are from \cite{rad1}.

The topological radical we are going to treat now is related to the joint
spectral radius. Let us call a Banach algebra $A$ \textit{compactly
quasinilpotent} if ${\rho}(M)=0$ for any precompact subset $M$ of $A$.

\begin{theorem}
In any Banach algebra $A$ there is a largest compactly quasinilpotent ideal
$\mathcal{R}_{cq}(A)$. The map $A\longmapsto\mathcal{R}_{cq}(A)$ is a
hereditary topological radical.
\end{theorem}

It is possible to give an individual test for an element to belong to
$\mathcal{R}_{cq}(A)$. It is formally similar to the known characterization of
the elements in the Jacobson radical.

\begin{theorem}
An element $a\in A$ belongs to $\mathcal{R}_{cq}(A)$ if and only if $\rho(aM)
= 0$ for any precompact set $M\subset A$.
\end{theorem}

The following result shows that $\mathcal{R}_{cq}(A)$ can be considered as
inessential when one calculates the joint spectral radius.

\begin{theorem}
\label{iness} Let $q=q_{\mathcal{R}_{cq}(A)}$ be the quotient map
$A\longrightarrow A/\mathcal{R}_{cq}(A)$. Then $\rho(M)=\rho(q(M))$ for each
precompact set $M\subset A$.
\end{theorem}

\subsection{The hypocompact radical}

We denote the set of all compact elements of a Banach algebra $A$ by $\mathcal{C}(A)$. Note that
$\mathcal{C}(A)$ is a semigroup ideal of $A$ not closed, in general, under addition, even for semisimple Banach
algebras. The following result follows easily from the Open Mapping Theorem.

\begin{lemma}
\label{lhc1} If $f:A\longrightarrow B$ is a continuous surjective
homomorphism of Banach algebras then $f(\mathcal{C}(A))\subset\mathcal{C}(B)$.
\end{lemma}

\begin{lemma}
\label{lhc2} Let $J$ be an ideal of $A$. If $\mathcal{C}(J)\neq0$ then
$J\cap\mathcal{C}(A)\neq0$.
\end{lemma}

\begin{proof}
It is easy to see that
\begin{equation}
\mathrm{W}_{ba}=\mathrm{L}_{b}\mathrm{W}_{a}\mathrm{R}_{b}=\mathrm{R}%
_{a}\mathrm{W}_{b}\mathrm{L}_{a}\label{prod}%
\end{equation}
for all $a,b\in A$. Hence if $a\in\mathcal{C}(J)$ then for any $b\in J$, the
operator $W_{ba}$ is compact on $A$. So
\[
\mathcal{C}(J)J\subset J\cap\mathcal{C}(A)
\]
and we are done if $\mathcal{C}(J)J\neq0$. On the other hand, if
$\mathcal{C}(J)J=0$ then
\[
\mathcal{C}(J)\subset\mathcal{C}(A)
\]
because $W_{a}(x)=a(xa)=0$ for any $a\in\mathcal{C}(J)$ and $x\in A$.
\end{proof}

A Banach algebra $A$ is called \textit{bicompact} if all operators
$\mathrm{L}_{a}\mathrm{R}_{b}$ ($a,b\in A$) are compact (in other words, $A$
is bicompact if it consists of mutually compact elements). An ideal of $A$ is
called \textit{bicompact} if it is bicompact as a Banach algebra.

\begin{lemma}
\label{lhc3}Let $A$ be a Banach algebra. Then

\begin{itemize}
\item If $a\in C(A)$ then the ideal $\mathcal{J}(a)$ generated by $a$ is bicompact.

\item If $J$ is a bicompact ideal of $A$ then the operator $\mathrm{L}%
_{a}\mathrm{R}_{b}$ is compact for every $a,b\in\operatorname*{span}(JA)$.
\end{itemize}
\end{lemma}

\begin{proof}
Both statements follow easily from (\ref{prod}).
\end{proof}

A Banach algebra $A$ is called \textit{hypocompact} if every non-zero quotient
$A/J$ has a non-zero compact element. An ideal is \textit{hypocompact} if it
is hypocompact as a Banach algebra.

Clearly each bicompact algebra is hypocompact. We will see that all
hypocompact algebras can be obtained by subsequent extensions of bicompact
ones. Let us prove first that the class of all hypocompact algebras is stable
under extensions.

\begin{lemma}
\label{lhcext} Let $J$ be an ideal of $A$. If $J$ and $A/J$ are hypocompact
then $A$ is hypocompact.
\end{lemma}

\begin{proof}
Let $I$ be an ideal of $A$. If $J\subset I$ then $A/I$ can be identified with
$(A/J)/(I/J)$, the quotient of a hypocompact algebra. By definition, the
latter contains non-zero compact elements, so does $A/I$.

Suppose now that $I$ does not contain $J$. Setting $K=J\cap I$ we have that
\[
\mathcal{C}(J/K)\neq0.
\]
By Lemma \ref{lhc2},
\[
J/K\cap\mathcal{C}(A/K)\neq0.
\]
Let $0\neq q_{K}(a)\in J/K\cap\mathcal{C}(A/K)$. Then
\[
a\notin I
\]
and
\[
q_{I}(a)\in\mathcal{C}(A/I)
\]
by Lemma \ref{lhc1}. Thus $A/I$ contains non-zero compact elements.
\end{proof}

\begin{proposition}
\label{phc1} Let $J$ be an ideal of a Banach algebra $A$. Then the following
conditions are equivalent.

\begin{itemize}
\item[(i)] $J$ is hypocompact.

\item[(ii)] For each continuous surjective homomorphism $f:A\longrightarrow
B$, either $f(J)=0$ or $f(J)\cap\mathcal{C}(B)\neq0$.

\item[(iii)] There is a transfinite increasing chain of ideals $(J_{\alpha
})_{\alpha\leq\gamma}$ in $A$ such that $J_{1}=0$, $J_{\gamma}=J$, and all
$J_{\alpha+1}/J_{\alpha}$ are bicompact.
\end{itemize}
\end{proposition}

\begin{proof}
(i)$\Longrightarrow$(ii) Let $I=\ker f$ and $K=I\cap J$.  If $J\subset I$
then
\[
f(J)=0.
\]
Otherwise there is a non-zero $a\in J/K\cap\mathcal{C}(A/K)$ by Lemma
\ref{lhc2}. Let $g:A/K\longrightarrow B$ be defined by
\[
g\left(  q_{K}\left(  b\right)  \right)  =f\left(  b\right)
\]
for every $b\in A$. Then $g$ is a continuous surjective homomorphism. Take
$b\in J$  such that
\[
a=q_{K}\left(  b\right)  .
\]
Then $f\left(  b\right)  =g\left(  a\right)  $ is a non-zero compact element
of $B$.

(ii)$\Longrightarrow$(iii) Let us consider all transfinite increasing chains $(J_{\alpha})_{\alpha\leq\beta}$
such that $J_{\alpha}\subset J$, $J_{\alpha +1}/J_{\alpha}$ is bicompact and $J_{\alpha}\neq J_{\alpha +1}$ for
any $\alpha<\beta$. Clearly these chains form a set because $A$ is a set. Order the set of chains by
\[
(J_{\alpha})_{\alpha\leq\beta_{1}}\prec(I_{\alpha})_{\alpha\leq\beta_{2}%
}\text{ if }\beta_{1}\leq\beta_{2}\text{ and }J_{\alpha}=I_{\alpha}\text{ for
}\alpha\leq\beta_{1}.
\]
It follows from the Zorn Lemma that there is a maximal element $(J_{\alpha
})_{\alpha\leq\gamma}$ in this set. If $J_{\gamma}\neq J$ then there is a
bicompact ideal $I$ of $J/J_{\gamma}$. Put
\[
J_{\gamma+1}=\{x\in J:q_{J_{\gamma}}(x)\in I\}.
\]
Then one can add $J_{\gamma+1}$ to the chain, in contradiction with its maximality.

(iii)$\Longrightarrow$(i) Let $I$ be an ideal of $J$. Let ${\alpha}$ be the
first  ordinal  for which $J_{\alpha}$ is not contained in $I$. Then
\[
q_{I}(J_{\alpha})\subset\mathcal{C}(J/I).
\]
It follows that
\[
\mathcal{C}(J/I)\neq0.
\]

\end{proof}

\begin{corollary}
\label{chcid} An ideal of a hypocompact Banach algebra is hypocompact.
\end{corollary}

\begin{proof}
Let $A$ be hypocompact and $J$ an ideal of $A$. Let $f:A\longrightarrow B$ be
a continuous surjective homomorphism, $I=\ker~f$ and $K=I\cap J$. Assuming
$f(J)\neq0$, we get that
\[
\mathcal{C}(J/I)\neq0,
\]
whence there is $a\in J$ such that%
\[
0\neq q_{K}(a)\in\mathcal{C}(A/K).
\]
It follows from Lemma \ref{lhc1} that
\[
0\neq q_{I}(a)\in\mathcal{C}(A/I),
\]
so $f(a)$ is a non-zero compact element of $B$ in $f(J)$.
\end{proof}

It follows easily from the definition that a quotient of a hypocompact algebra
by an ideal is hypocompact. So Lemma \ref{lhcext} can be stated in the
classical form of Three Subspaces Theorem:

\textit{Let }$A$ \textit{be a Banach algebra. The following are equivalent.}

\begin{itemize}
\item $A$\textit{ is hypocompact. }

\item $J$\textit{ and }$A/J$\textit{ are hypocompact for every ideal }%
$J$\textit{ of }$A$\textit{.}

\item $J$\textit{ and }$A/J$\textit{ are hypocompact for some ideal }%
$J$\textit{ of }$A$\textit{.}
\end{itemize}

As a consequence, we have the following.

\begin{corollary}
\label{lhclar} In any Banach algebra there is the largest hypocompact ideal.
\end{corollary}

\begin{proof}
Let $J$ be the closed linear span of the union of all hypocompact ideals of a
Banach algebra $A$. We have to prove that the ideal $J$ is hypocompact.

By Proposition \ref{phc1}, it suffices to prove that if $f:A\longrightarrow B$
is a continuous surjective homomorphism with $f(J)\neq0$ then
\[
f(J)\cap\mathcal{C}(B)\neq0.
\]
But if $f(J)\cap\mathcal{C}(B)=0$ then
\[
f(I)=0
\]
for each hypocompact ideal $I$ of $A$. Hence
\[
f(J)=0.
\]

\end{proof}

The largest hypocompact ideal of a Banach algebra $A$ will be denoted by
$\mathcal{R}_{\mathrm{hc}}(A)$.

\begin{lemma}
\label{lhcint} If $J$ is an ideal of $A$ then $\mathcal{R}_{\mathrm{hc}%
}(J)=J\cap\mathcal{R}_{\mathrm{hc}}(A)$.
\end{lemma}

\begin{proof}
By Corollary \ref{chcid}, the ideal $J\cap\mathcal{R}_{\mathrm{hc}}(A)$ of $J$
is hypocompact so it is contained in $\mathcal{R}_{\mathrm{hc}}(J)$. We have
to prove the converse inclusion.

Let $I=\operatorname*{span}(A^{1}\mathcal{R}_{\mathrm{hc}}(J)A^{1})$ be the
ideal of $A$ generated by $\mathcal{R}_{\mathrm{hc}}(J)$. Then
\begin{align*}
I^{3}  &  =\operatorname*{span}\left(  \left(  A^{1}\mathcal{R}_{\mathrm{hc}%
}(J)A^{1}A^{1}\right)  \mathcal{R}_{\mathrm{hc}}(J)\left(  A^{1}%
A^{1}\mathcal{R}_{\mathrm{hc}}(J)A^{1}\right)  \right) \\
&  \subset\operatorname*{span}\left(  J\mathcal{R}_{\mathrm{hc}}(J)J\right)
\subset\mathcal{R}_{\mathrm{hc}}(J).
\end{align*}
Hence $I^{3}$ is a hypocompact ideal. But $I/I^{3}$ is bicompact because
$L_{a}R_{b}=0$ for every $a,b\in I/I^{3}$. By Lemma \ref{lhcext}, $I$ is
hypocompact. Then
\[
I\subset\mathcal{R}_{\mathrm{hc}}(A)
\]
and
\[
\mathcal{R}_{\mathrm{hc}}(J)\subset\mathcal{R}_{\mathrm{hc}}(A).
\]

\end{proof}

\begin{lemma}
\label{lhccel} The algebra $A/\mathcal{R}_{\mathrm{hc}}(A)$ has no hypocompact
ideals and compact elements.
\end{lemma}

\begin{proof}
If $J$ is a hypocompact ideal of $A/\mathcal{R}_{\mathrm{hc}}(A)$ then, by
Lemma \ref{lhcext}, its preimage $J_{1}=\left\{  x\in A:q_{\mathcal{R}%
_{\mathrm{hc}}(A)}(x)\in J\right\}  $ is a hypocompact ideal of $A$ strictly
containing $\mathcal{R}_{\mathrm{hc}}(A)$, a contradiction.

By Lemma \ref{lhc2}, if $A/\mathcal{R}_{\mathrm{hc}}(A)$ has compact elements
then it has bicompact ideals.
\end{proof}

\begin{theorem}
\label{thcrad} The map $A\longmapsto\mathcal{R}_{\mathrm{hc}}(A)$ is a
hereditary topological radical.
\end{theorem}

\begin{proof}
(R1) Since $A/\mathcal{R}_{\mathrm{hc}}(A)$ has no hypocompact ideals, we have
that%
\[
\mathcal{R}_{\mathrm{hc}}(A/\mathcal{R}_{\mathrm{hc}}(A))=0.
\]

(R2) Let $f:A\rightarrow B$ be a continuous surjective homomorphism. Denote
$q_{\mathcal{R}_{\mathrm{hc}}(B)}$ by $q$, for brevity. Clearly $q{\circ}f$ is
a continuous surjective homomorphism of $A$ to $B/\mathcal{R}_{\mathrm{hc}%
}(B)$. Since $\mathcal{R}_{\mathrm{hc}}(A)$ is hypocompact, $q{\circ
}f(\mathcal{R}_{\mathrm{hc}}(A))$ is either zero or contains a compact element
of $B$. But the latter is impossible by Lemma \ref{lhccel}. Hence
\[
q{\circ}f(\mathcal{R}_{\mathrm{hc}}(A))=0
\]
and
\[
f(\mathcal{R}_{\mathrm{hc}}(A))\subset\mathcal{R}_{\mathrm{hc}}(B).
\]

(R3) This was proved in Lemma \ref{lhcint}.
\end{proof}

\subsection{The radical $\operatorname*{Rad}\wedge\mathcal{R}_{\mathrm{hc}}$}

Starting with a family of radicals, one can obtain some new ones. The
following construction has a well known analogue in the theory of algebraic
radicals on rings. Let $P_{1}$ and $P_{2}$ be hereditary topological radicals.
For any Banach algebra $A$, put
\[
P_{0}(A)=P_{1}(A)\cap P_{2}(A).
\]

\begin{theorem}
\label{wedge} The map $A\longmapsto P_{0}(A)$ is a hereditary topological radical.
\end{theorem}

\begin{proof}
Set $D=A/P_{0}(A)$. There is a natural surjective homomorphism $q_{1}%
:D\rightarrow A/P_{1}(A)$ defined by
\[
q_{1}\left(  a/P_{0}(A)\right)  =a/P_{1}(A).
\]
Since $P_{1}(A/P_{1}(A))=0$, we have that%
\[
q_{1}(P_{1}(D))=0
\]
and
\[
P_{1}(D)\subset\ker(q_{1}).
\]
Therefore, if an element $a/P_{0}(A)$ of $D$ belongs to the kernel of the
homomorphism $q_{1}$, then $a\in P_{1}(A)$. We get that
\[
P_{1}(D)\subset\{a/P_{0}(A):a\in P_{1}(A)\}.
\]
Similarly,
\[
P_{2}(D)\subset\{a/P_{0}(A):a\in P_{2}(A)\}.
\]
Hence
\[
P_{0}(D)\subset\left\{  a/P_{0}(A):a\in P_{1}(A)\cap P_{2}(A)\right\}
=P_{0}(A)/P_{0}(A).
\]
In other words,
\[
P_{0}(D)=0.
\]
We proved that $P_{0}$ satisfies condition (R1).

If $f:A\longrightarrow B$ is a continuous surjective homomorphism then
\[
f(P_{0}(A))\subset f\left(  P_{1}\left(  A\right)  \right)  \subset P_{1}(B)
\]
and, similarly,
\[
f(P_{0}(A))\subset P_{2}(B).
\]
Hence
\[
f(P_{0}(A))\subset P_{0}(B),
\]
therefore $P_{0}$ satisfies (R2).

For an ideal $J\subset A$, one has
\[
P_{0}(J)=P_{1}(J)\cap P_{2}(J)=\left(  P_{1}(A)\cap J\right)  \cap\left(
P_{2}(A)\cap J\right)  =P_{0}(A)\cap J.
\]
We proved (R3).
\end{proof}

The radical $P_{0}$ constructed in the previous theorem is denoted by
$P_{1}\wedge P_{2}$. We will consider the radical $\operatorname*{Rad}%
\wedge\mathcal{R}_{\mathrm{hc}}$.

Let us introduce a standard order in the class of all topological radicals by
the rule
\[
P_{1}\leq P_{2}%
\]
if $P_{1}(A)\subset P_{2}(A)$ for every Banach algebra $A$. One can show that
$P_{1}\wedge P_{2}$\textit{ is the largest hereditary topological radical }%
$P$\textit{ having the property }$P\leq P_{1}$\textit{ and }$P\leq P_{2}%
$\textit{,} but we needn't it here.

The following result shows in particular that for compact (more generally, for
hypocompact) algebras the radical $\mathcal{R}_{\mathrm{cq}}$ coincides with
the Jacobson radical.

\begin{theorem}
\label{order} $\operatorname*{Rad}\wedge\mathcal{R}_{\mathrm{hc}}%
\leq\mathcal{R}_{\mathrm{cq}}$.
\end{theorem}

\begin{proof}
Let us first show that each bicompact Jacobson radical algebra $A$ is
compactly quasinilpotent. Indeed, if $M$ is a precompact family in $A$ then
\[
r(M)=0
\]
because $A$ consists of quasinilpotent elements. For $N=\mathrm{L}%
_{M}\mathrm{R}_{M}$, we see from Lemma \ref{pass} that
\[
r(N)=0.
\]
Since $N$ is a precompact family of compact operators,
\[
\rho(N)=0
\]
by the Berger-Wang formula. Again, we obtain from Lemma \ref{pass} that
\[
\rho(M)=0.
\]

Let now $A$ be an arbitrary Banach algebra and $J=\operatorname*{Rad}%
(A)\cap\mathcal{R}_{\mathrm{hc}}(A)$. Since $J$ is a hypocompact ideal then,
by Proposition \ref{phc1}, there is a transfinite increasing chain
$(J_{\alpha})_{\alpha\leq\gamma}$ of ideals such that
\[
J_{\gamma}=J
\]
and all $J_{\gamma+1}/J_{\gamma}$ are bicompact. Since all bicompact Jacobson
radical ideals are compactly quasinilpotent, all $J_{\gamma+1}/J_{\gamma}$ are
$\mathcal{R}_{\mathrm{cq}}$-radical. By the transfinite extension property
(see Lemma \ref{transext}), $J$ is $\mathcal{R}_{\mathrm{cq}}$-radical.
\end{proof}

\section{Main results}

\subsection{Mixed\ GBWF}

We will prove  for any precompact set $M$ of elements in a Banach algebra $A$
that
\begin{equation}
\rho(M)=\max\{\rho^{\chi}(M),r(M)\},\label{e2}%
\end{equation}
where as above we set
\[
\rho^{\chi}(M)=\rho_{\chi}(\mathrm{L}_{M}\mathrm{R}_{M})^{1/2}.
\]

Note that it suffices to prove this result under the assumption that $A$ is
generated by $M$ as a Banach algebra. Indeed, $\rho(M)$ and $r(M)$ do not
change if calculated in the closed subalgebra $B=\mathcal{A}(M)$ generated by
$M$. The value $\rho^{\chi}(M)$ in this case cannot increase because the
multiplication operators act on a smaller space. But the nontrivial inequality
in (\ref{e2}) is only $\leq$.

So we may assume in what follows that $A=\mathcal{A}(M)$. A semigroup $G$ of elements of a Banach algebra is
called a \textit{Radjavi semigroup} ($R$\textit{-semigroup }for brevity) it ${\lambda}a\in G$ for every $a\in G$
and $\lambda\geq0$.

Let $G=\mathcal{S}(M)$ be the semigroup generated by a set $M$ of operators.
Then
\[
G=\cup_{n=1}^{\infty}M^{n}.
\]
An operator $T\in M^{n}$ is called \textit{leading} (more precisely,
$n$-\textit{leading}) if
\[
\left\Vert T\right\Vert \geq\left\Vert S\right\Vert
\]
for all $S\in\cup_{k\leq n}M^{k}$. Note that an operator may be in the
different $M^{n}$'s, this justifies the more precise term `$n$-leading
operator', but we write just 'leading' when it is clear which $n$ is meant.
A \textit{leading sequence} in $G$ is a sequence that consists of leading
operators $T_{k}\in M^{n\left(  k\right)  }$ for $n(k)\rightarrow\infty$.

Let $\mathcal{S}_{+}\left(  M\right)  $ be the $R$-semigroup generared by $M$.
Clearly
\[
\mathcal{S}_{+}\left(  M\right)  =\mathbb{R}_{+}\mathcal{S}(M),
\]
where $\mathbb{R}_{+}=\{t\in\mathbb{R}:t\geq0\}$.

The following lemma was proved in \cite[see Theorem 6.10]{ShT2000}.

\begin{lemma}
\label{L3} Let $N$ be a precompact set of operators. Suppose that $\rho_{\chi
}(N)<\rho(N)=1$ and $\mathcal{S}(N)$ is unbounded. Then there is a sequence
$T_{n}\in\mathcal{S}_{+}(N)$, with $\Vert T_{n}\Vert=1$, converging to a
compact operator $T$. Moreover, to obtain such a sequence $T_{n}$ it suffices
to take any leading sequence $S_{n}$ in $\mathcal{S}(N)$ and choose a
convergent subsequence from $S_{n}/\Vert S_{n}\Vert$ (it always exists).
\end{lemma}

\begin{lemma}
\label{L4} Let $A$ be a Banach algebra, $M\subset A$ be precompact and
$N=\mathrm{L}_{M}\mathrm{R}_{M}$. Suppose that $A=\mathcal{A}\left(  M\right)
$, $\rho_{\chi}(N)<\rho(N)=1$ and $\mathcal{S}(N)$ is unbounded. Then the
closure $\overline{\mathcal{S}_{+}(N)}$ contains a non-zero compact operator
$T$ such that

\begin{itemize}
\item[(i)] $\mathrm{L}_{Th}\mathrm{R}_{Tg}$ is compact for every elements
$h,g$ of $A$.

\item[(ii)] If also $r(N)<1$ then $T(A)\subset\operatorname*{Rad}(A)$.
\end{itemize}
\end{lemma}

\begin{proof}
All elements in $\mathcal{S}_{+}(N)$ are of the form
\[
P={\lambda}\mathrm{L}_{a}\mathrm{R}_{b},
\]
where $a,b\in\mathcal{S}(M)$ and ${\lambda\geq}0$. For brevity, we will denote
$P^{\circ}$ for any $P$ by
\[
P^{\circ}={\lambda}\mathrm{L}_{b}\mathrm{R}_{a}.
\]
Note that $P^{\circ}$ may be not uniquely determined by $P$, but the equality
\[
\mathrm{L}_{Ph}\mathrm{R}_{Pg}=P\mathrm{L}_{h}\mathrm{R}_{g}P^{\circ}%
\]
holds independently of the choice of $P^{\circ}$, for every $h,g\in M$.

Let $\left(  S_{n}\right)  $ be a leading sequence  in $\mathcal{S}(N)$. For every $S_{n}\in N^{m\left(
n\right)  }$, the operator $S_{n}^{\circ}$ can also be chosen in $N^{m\left(  n\right)  }$, so we may assume
that
\begin{equation}
\Vert S_{n}\Vert\geq\Vert S_{n}^{\circ}\Vert\label{e5}%
\end{equation}
for all $n$. By Lemma \ref{L3}, we may  choose a sequence of operators
\[
T_{n}=S_{k_{n}}/\Vert S_{k_{n}}\Vert
\]
that tends to a compact operator $T$. Note that  all operators
\[
T_{n}^{\circ}=S_{k_{n}}^{\circ}/\Vert S_{k_{n}}\Vert
\]
are contractive by (\ref{e5}). Now for any $h,g\in A$, we have
\[
\mathrm{L}_{Th}\mathrm{R}_{Tg}=\lim_{n\rightarrow\infty}\mathrm{L}{_{T_{n}%
h}\mathrm{R}_{T_{n}g}}=\lim_{n\rightarrow\infty}T_{n}\mathrm{L}_{h}%
\mathrm{R}_{g}T_{n}^{\circ}=\lim_{n\rightarrow\infty}T\mathrm{L}_{h}%
\mathrm{R}_{g}T_{n}^{\circ}.
\]
Hence the operator $\mathrm{L}_{Th}\mathrm{R}_{Tg}$ is a limit of compact
operators, so it is compact. Part (i) is proved.

Let $r(N)<1$, and let us now prove that $u(Tx)v$ is quasinilpotent for every
$u,v,x\in\mathcal{S}(M)$. By our construction,
\[
T=\lim_{n\rightarrow\infty}{\lambda}_{k_{n}}S_{k_{n}},
\]
where ${\lambda}_{k_{n}}=\Vert S_{k_{n}}\Vert^{-1}\rightarrow0$ as
$n\rightarrow\infty$, $S_{k_{n}}=\mathrm{L}_{a_{n}}\mathrm{R}_{b_{n}}$ for
some $a_{n},b_{n}\in\mathcal{S}(M)$. Since%
\[
\mathrm{L}_{\mathcal{S}(M)}\mathrm{R}_{\mathcal{S}(M)}=\mathcal{S}\left(
N\right)  ,
\]
we have that%
\[
\mathrm{W}_{ua_{n}xb_{n}v}=\left(  \mathrm{L}_{u}\mathrm{R}_{v}\right)
\left(  \mathrm{L}_{a_{n}}\mathrm{R}_{b_{n}}\right)  \mathrm{W}_{x}\left(
\mathrm{L}_{b_{n}}\mathrm{R}_{a_{n}}\right)  \left(  \mathrm{L}_{v}%
\mathrm{R}_{u}\right)  \in\mathcal{S}\left(  N\right)  .
\]
Since $r(N)<1$ implies that $\{\rho(S):S\in\mathcal{S}(N)\}$ is bounded, we
obtain that
\[
\rho(u(Tx)v)=\lim_{n\rightarrow\infty}{\lambda}_{k_{n}}\rho(ua_{n}%
xb_{n}v)=\lim_{n\rightarrow\infty}{\lambda}_{k_{n}}\rho(\mathrm{W}%
_{ua_{n}xb_{n}v})^{1/2}=0,
\]
Thus we see that the set $\mathcal{S}_{+}(M)(Tx)\mathcal{S}_{+}(M)$ consists
of mutually compact quasinilpotent elements of $A$ for every $x\in
\mathcal{S}(M)$. So does the closure $\overline{\mathcal{S}_{+}%
(M)(Tx)\mathcal{S}_{+}(M)}$. Since
\begin{align*}
Tx  & \in\overline{\mathcal{S}_{+}(N)}\mathcal{S}(M)=\overline{\mathrm{L}%
_{\mathcal{S}_{+}(M)}\mathrm{R}_{\mathcal{S}_{+}(M)}}\mathcal{S}%
(M)\subset\overline{\mathcal{S}_{+}(M)\mathcal{S}(M)\mathcal{S}_{+}(M)}\\
& \subset\overline{\mathcal{S}_{+}(M)},
\end{align*}
the set $\overline{\mathcal{S}_{+}(M)(Tx)\mathcal{S}_{+}(M)}$ is a semigroup.
By Lemma \ref{semig}, its closed linear span $J$ also consists of compact
quasinilpotent elements. Since $J$ coincides with the ideal $\overline
{A(Tx)A}$ of $A$, we have that
\[
A(Tx)A\subset\operatorname*{Rad}(A),
\]
whence, by the quasi-regular characterization of the Jacobson radical,
\[
A(Tx)\subset\operatorname*{Rad}(A)
\]
and also
\[
Tx\in\operatorname*{Rad}(A).
\]
Since $A=\operatorname*{span}\mathcal{S}(M)$, we obtain that
\[
T(A)\subset\operatorname*{Rad}(A).
\]

\end{proof}

Let us call any closed bicompact ideal that consists of quasinilpotent
operators a $qb$\textit{-ideal}. The above lemma implies the following result.

\begin{corollary}
\label{imcor} If  $~\max\left\{  \rho^{\chi}\left(  M\right)  ,r\left(
M\right)  \right\}  <\rho(M)=1$ and the semigroup $\mathcal{S}(\mathrm{L}%
_{M}\mathrm{R}_{M})$ is unbounded then ${\mathcal{A}(M)}$ has a non-zero $qb$-ideal.
\end{corollary}

\begin{proof}
Indeed, every ideal $J$ generated by $Tx\in{\mathcal{A}(M)}$ is a $qb$-ideal.
\end{proof}

\begin{lemma}
\label{L5} If $\mathcal{A}{(M)}$ has no non-zero $qb$-ideals then the equality
\emph{(\ref{e2})} holds.
\end{lemma}

\begin{proof}
Suppose that (\ref{e2}) fails. We may assume that
\[
\rho(M)=1.
\]
Let $N=\mathrm{L}_{M}\mathrm{R}_{M}$. Then we have that%
\[
\rho(N)=1
\]
by Lemma \ref{pass}.

If the semigroup $\mathcal{S}(N)$ is bounded, then
\[
\max\left\{  \rho_{\chi}(N),r\left(  N\right)  \right\}  =1
\]
holds by \cite[Proposition 9.6]{ShT2000}. If $\rho_{\chi}(N)=1$ then
$\rho^{\chi}(M)=1$. Otherwise $r(N)=1$ and $r(M)=1$ by Lemma \ref{pass}. In
both of the cases (\ref{e2}) holds, a contradiction.

So $\mathcal{S}(N)$ is unbounded. Then $\mathcal{A}\left(  M\right)  $ has a
non-zero $qb$-ideal by Corollary \ref{imcor}. This contradicts to our assumptions.
\end{proof}

\begin{theorem}
\label{rho"} Let $A$ be a Banach algebra. The equality \emph{(\ref{e2})} holds
for each precompact subset $M$ of $A$.
\end{theorem}

\begin{proof}
Recall that we may assume that $A=\mathcal{A}(M)$. Let $J=\operatorname*{Rad}%
(A)\cap\mathcal{R}_{\mathrm{hc}}(A)$. Since
\[
J\subset\mathcal{R}_{\mathrm{cq}}(A)
\]
by Theorem \ref{order},we obtain that
\[
\rho(M)=\rho(M/J).
\]
by Theorem \ref{iness}. Furthermore, the algebra $A/J$ has no $qb$-ideals.
Indeed, if $I$ is such an ideal, and $U$ is its preimage in $A$, then it is
evident that%
\[
U\subset\operatorname*{Rad}(A),
\]
and that
\[
U\subset\mathcal{R}_{\mathrm{hc}}(A)
\]
by the extension property of radicals. Hence
\[
U\subset J
\]
and, as a consequence,
\[
I=0.
\]
Taking into account that $A/J=\mathcal{A}\left(  M/J\right)  $, and applying
Lemma \ref{L5}, we have that
\[
\rho(M)=\rho(M/J)=\max\{\rho^{\chi}(M/J),r(M/J)\}\leq\max\{\rho^{\chi
}(M),r(M)\}.
\]
The converse inequality is evident.
\end{proof}

\subsection{ Operator GBWF}

Now we can prove (\ref{ggbwf}).

\begin{theorem}
\label{top} If $M\subset\mathcal{B}(\mathcal{X})$ is precompact then
\[
\rho(M)=\max\{\rho_{\chi}(M),r(M)\}.
\]

\end{theorem}

\begin{proof}
By Lemma \ref{ineq},
\[
\Vert\mathrm{L}_{M}\mathrm{R}_{M}\Vert_{\chi}\leq16\Vert M\Vert_{\chi}\Vert
M\Vert.
\]
Changing $M$ by $M^{n}$, taking $n$-th roots and limits as $n\rightarrow
\infty$, we obtain that%
\[
\rho^{\chi}\left(  M\right)  ^{2}=\rho_{\chi}(\mathrm{L}_{M}\mathrm{R}%
_{M})\leq\rho_{\chi}(M)\rho(M).
\]
Applying this in (\ref{e2}), we get that
\[
\rho(M)\leq\max\{\rho_{\chi}(M)^{1/2}\rho(M)^{1/2},r(M)\},
\]
whence
\[
\rho(M)\leq\max\{\rho_{\chi}(M)^{1/2}\rho(M)^{1/2},r(M)^{1/2}\rho(M)^{1/2}\}.
\]
It follows from this that
\[
\rho(M)^{1/2}\leq\left(  \max\{\rho_{\chi}(M),r(M)\}\right)  ^{1/2},
\]
and we are done, because the converse is evident.
\end{proof}

\subsection{Banach algebraic GBWF}

Our next aim is to prove for any Banach algebra $A$ and a precompact subset $M\subset A$, that
\begin{equation}
\rho(M)=\max\{\rho(M/\mathcal{R}_{hc}(A)),r(M)\}.\label{alg}%
\end{equation}
It will be more convenient for us to prove (\ref{alg}) in the following form:
\begin{equation}
\rho(M)=\max\{\rho(M/J),r(M)\}\label{id}%
\end{equation}
for any hypocompact ideal $J$ of $A$.

We will begin with the case that $J$ is bicompact.

\begin{lemma}
\label{L6} Let $J$ be a bicompact ideal of $A$. Then
\begin{equation}
{\rho}_{e}(\mathrm{L}_{M}\mathrm{R}_{M})\leq\rho(M/J)\rho(M).\label{roe}%
\end{equation}

\end{lemma}

\begin{proof}
Let us prove first the inequality
\begin{equation}
\Vert\mathrm{L}_{M}\mathrm{R}_{M}\Vert\leq3\Vert M/J\Vert\Vert M\Vert
\label{norm}%
\end{equation}
Let $a,b\in M$, $\varepsilon>0$. Choose $u,v\in J$ such that
\[
\max\left\{  \Vert a-u\Vert,\Vert b-v\Vert\right\}  <\Vert M/J\Vert
+\varepsilon.
\]
In particular, we have that%
\[
\left\Vert u\right\Vert <\left\Vert a\right\Vert +\left\Vert a-u\right\Vert
\leq\left\Vert M\right\Vert +\Vert M/J\Vert+\varepsilon\leq2\left\Vert
M\right\Vert +\varepsilon.
\]
Then we obtain that
\begin{align*}
\left\Vert \mathrm{L}_{a}\mathrm{R}_{b}\right\Vert _{e} &  \leq\left\Vert
\mathrm{L}_{a}\mathrm{R}_{b}-\mathrm{L}_{u}\mathrm{R}_{v}\right\Vert
=\left\Vert \mathrm{L}_{a-u}\mathrm{R}_{b}+\mathrm{L}_{u}\mathrm{R}%
_{b-v}\right\Vert \\
&  \leq(\Vert M/J\Vert+\varepsilon)\Vert M\Vert+(2\Vert M\Vert+\varepsilon
)(\Vert M/J\Vert+\varepsilon)\\
&  \leq(\Vert M/J\Vert+\varepsilon)(3\Vert M\Vert+\varepsilon),
\end{align*}
and it remains to take $\varepsilon\rightarrow0$ and supremum over all $a,b\in
M$.

To obtain (\ref{roe}), change in (\ref{norm}) $M$ by $M^{n}$, take $n$-th
roots and $n\rightarrow\infty$.
\end{proof}

\begin{corollary}
\label{bi} The equality \emph{(\ref{id})} holds for each bicompact ideal $J$.
\end{corollary}

\begin{proof}
It follows from (\ref{roe}) that
\[
\rho^{\chi}(M)\leq{\rho}_{e}(\mathrm{L}_{M}\mathrm{R}_{M})^{1/2}\leq
\rho(M/J)^{1/2}\rho(M)^{1/2},
\]
whence by (\ref{e2}),
\begin{align*}
\rho(M)  & =\max\left\{  \rho^{\chi}(M),r\left(  M\right)  \right\}  \\
& \leq\max\{\rho(M/J)^{1/2}\rho(M)^{1/2},r(M)^{1/2}\rho(M)^{1/2}\}
\end{align*}
and (\ref{id}) follows immediately.
\end{proof}

\begin{lemma}
\label{step} Let $I,K$ be ideals of $A$ and $I\subset K$. If $K/I$ is
bicompact and \emph{(\ref{id})} holds for $J=I$ then it holds for $J=K$.
\end{lemma}

\begin{proof}
The isomorphism $A/K\rightarrow(A/I)/(K/I)$ implies
\[
\rho(M/K)=\rho((M/I)/(K/I)),
\]
whence
\begin{align*}
\rho(M) &  =\max\{\rho(M/I),r(M)\}=\max\{\max\{\rho
((M/I)/(K/I)),r(M/I)\},r(M)\}\\
&  \leq\max\{\rho(M/K),r(M)\}.
\end{align*}
The converse inequality is trivial.
\end{proof}

\begin{lemma}
\label{cont} If $J=\overline{\bigcup J_{\alpha}}$, where $(J_{\alpha})$ is an
increasing net of closed ideals of a Banach algebra $A$, then, for a
precompact subset $M\subset A$,%
\begin{equation}
\label{econt1}\Vert M/J\Vert=\lim_{\alpha}\Vert M/J_{\alpha}\Vert=\inf
_{\alpha}\Vert M/J_{\alpha}\Vert
\end{equation}
and
\begin{equation}
\rho(M/J)=\lim_{\alpha}\rho(M/J_{\alpha})=\inf_{\alpha}\rho(M/J_{\alpha}).
\label{econt}%
\end{equation}

\end{lemma}

\begin{proof}
We have
\[
\left\Vert M/J\right\Vert \leq\left\Vert M/J_{\alpha}\right\Vert
\]
for every $\alpha$, whence
\[
\rho(M/J)=\inf_{n}\left\Vert M^{n}/J\right\Vert ^{1/n}\leq\inf_{\alpha}%
\inf_{n}\left\Vert M^{n}/J_{\alpha}\right\Vert ^{1/n}=\inf_{\alpha}%
\rho(M/J_{\alpha}).
\]
and also%
\[
\left\Vert M/J\right\Vert \leq\inf_{\alpha}\left\Vert M/J_{\alpha}\right\Vert
\leq\underset{\alpha}{\,\lim\inf}\left\Vert M/J_{\alpha}\right\Vert .
\]

On the other hand, it is easy to see that for any $a\in M$ and $\varepsilon>0$
there is $\alpha=\alpha(a,\varepsilon)$ with
\begin{equation}
\left\Vert a/J_{\alpha}\right\Vert \leq\left\Vert a/J\right\Vert
+\varepsilon\leq\left\Vert M/J\right\Vert +\varepsilon.\label{econt2}%
\end{equation}
Take a finite subset $N$ of $M$ with $\operatorname*{dist}(b,N)\leq
\varepsilon$ for every $b\in M$. It is clear that, for every $c\in A$,
\[
\left\Vert c/J_{\beta}\right\Vert \leq\left\Vert c/J_{\alpha}\right\Vert
\]
if $\alpha<\beta$. So, choosing $\gamma\geq\max\left\{  \alpha(a,\varepsilon
):a\in N\right\}  $, we obtain from (\ref{econt2}) that
\[
\operatorname*{dist}(b/J_{\gamma},N/J_{\gamma})\leq\operatorname*{dist}%
(b,N)\leq\varepsilon
\]
for every $b\in M$, and so
\[
\left\Vert M/J_{\gamma}\right\Vert \leq\left\Vert N/J_{\gamma}\right\Vert
+\varepsilon\leq\left\Vert M/J\right\Vert +2\varepsilon.
\]
Therefore
\begin{equation}
\inf_{\alpha}\left\Vert M/J_{\alpha}\right\Vert \leq\underset{\alpha}%
{\,\lim\sup\,}\left\Vert M/J_{\alpha}\right\Vert \leq\left\Vert M/J\right\Vert
,\label{econt3}%
\end{equation}
whence (\ref{econt1}) holds. Take $n\in\mathbb{N}$ such that%
\[
\left\Vert M^{n}/J\right\Vert ^{1/n}\leq\rho(M/J)+\varepsilon.
\]
Then, by (\ref{econt3}) applied to $M^{n}$,
\begin{align*}
\inf_{\alpha}\rho(M/J_{\alpha}) &  \leq\lim\sup\,\,\rho(M/J_{\alpha}%
)\leq\underset{\alpha}{\,\lim\sup\,}\Vert M^{n}/J_{\alpha}\Vert^{1/n}%
\leq\left\Vert M^{n}/J\right\Vert ^{1/n}\\
&  \leq\rho(M/J)+\varepsilon.
\end{align*}
Therefore (\ref{econt}) holds.
\end{proof}

Now we can finish the proof of (\ref{id}).

\begin{theorem}
The equality \emph{(\ref{id})}\textrm{ }holds for every hypocompact ideal $J$.
\end{theorem}

\begin{proof}
Indeed, there is a transfinite chain $\{J_{\alpha}\}_{\alpha\leq\beta}$ of
closed ideals such that $J_{0}=0$, $J_{\beta}=J$, and all $J_{\alpha
+1}/J_{\alpha}$ are bicompact. Suppose that $\gamma$ is the least $\alpha$,
for which (\ref{id}) fails. It cannot be a limit ordinal because of Lemma
\ref{cont} and cannot have a predecessor because of Lemma \ref{step}.
Therefore (\ref{id}) holds for all $\alpha$.
\end{proof}

\subsection{Applications to continuity of the joint spectral radius}

Since the operator GBWF is now proved in full generality we may remove the
restriction of weak compactness in the applications to the continuity of joint
spectral radius which were obtained in \cite[Corollary 4.6]{ShT2002}.

Recall \cite[Proposition 3.1]{ShT2000} that \textit{the joint spectral radius
is an upper semicontinuous function of a bounded subset} $M$ \textit{of a
Banach algebra}. This means that
\[
\underset{n\rightarrow\infty}{\limsup\,\,}\rho(M_{n})\leq\rho(M)
\]
if a sequence $M_{n}$ of bounded subsets tends to $M$ in the sense that the
Hausdorff distance between $M_{n}$ and $M$ tends to zero.

\begin{quotation}
{\small Indeed, we have that }%
\[
M_{n}^{m}\rightarrow M^{m}%
\]
{\small as }$n\rightarrow\infty$ {\small  for every }$m${\small , whence }%
\[
\Vert M_{n}^{m}\Vert^{1/m}\rightarrow\Vert M^{m}\Vert^{1/m}%
\]
{\small as }$n\rightarrow\infty${\small . Since }$\rho(M_{n})=\rho(M_{n}%
^{m})^{1/m}\leq\Vert M_{n}^{m}\Vert^{1/m}${\small , we see that }%
\[
\underset{n\rightarrow\infty}{\limsup\,\,}\rho(M_{n})\leq\Vert M^{m}%
\Vert^{1/m}.
\]
{\small Taking }$m\rightarrow\infty${\small , we get that }%
\[
\underset{n\rightarrow\infty}{\limsup\,\,}\rho(M_{n})\leq\rho(M).
\]

\end{quotation}

We say that a set $M$ of elements in a Banach algebra $\mathcal{A}$ is a
\textit{point of continuity} for the joint spectral radius if $\rho
(M_{n})\rightarrow\rho(M)$ for any sequence $M_{n}$ of bounded sets tending to
$M$.

It is well known that if the norm of an operator $T$ is more than its
essential norm then $T$ is a point of continuity of the (usual) spectral
radius. The following result establishes the same for precompact families of operators.

\begin{corollary}
\label{appl} Let $M$ be a precompact set of operators on a Banach set
$\mathcal{X}$. If $\rho_{\chi}(M)<\rho(M)$ then $M$ is a point of continuity
of the joint spectral radius.
\end{corollary}

\begin{proof}
Let $M_{n}$ tend to $M$. Since
\[
\underset{n\rightarrow\infty}{\limsup\,\,}\rho(M_{n})\leq\rho(M),
\]
we have only to prove that
\[
\underset{n\rightarrow\infty}{\liminf\,}\rho(M_{n})\geq\rho(M).
\]
Suppose the contrary. Multiplying by a scalar and changing $M_{n}$ by a
subsequence, we may assume that
\[
\rho(M_{n})\rightarrow\alpha<1<\rho(M)
\]
and
\[
\rho_{\chi}(M)<1.
\]

It follows from the formula (\ref{ggbwf}) that
\[
\rho(M)=r(M).
\]
Hence
\[
\sup\{\rho(T):T\in M^{k}\}>1
\]
for some $k$. This means that there is an operator $T\in M^{k}$ with
\[
\rho(T)>1.
\]
Note that
\[
\rho_{\chi}(T)\leq1.
\]
Indeed, since $\rho_{\chi}(M)<1$ then
\[
\Vert M^{n}\Vert_{\chi}<1
\]
for sufficiently large $n$. Hence
\[
\Vert T^{n}\Vert_{\chi}\leq\Vert M^{nk}\Vert_{\chi}<1
\]
and it remains to take the $n$-th roots.

Since for operators (one-element families) the numbers $\rho_{\chi}$ and
$\rho_{e}$ coincide, we conclude that
\[
\rho_{e}(T)<\rho(T).
\]
By our assumptions, there are $T_{n}\in M_{n}^{k}$ such that
\[
T_{n}\rightarrow T.
\]
Since $T$ is a point of continuity of the usual spectral radius,
\[
\rho(T_{n})\rightarrow\rho(T).
\]
But this is impossible because
\[
\rho(T_{n})\leq\rho(M_{n}^{k})=\rho(M_{n})^{k}\rightarrow{\alpha}^{k}<1.
\]

\end{proof}

\textit{Acknowledgement.} The results of the paper, as well as some other
results on topological radicals of Banach algebras, were announced in
\cite{ShT2001}. At about that time the authors realized that the theory of
topological radicals admits a systematic treatment in a much more general
setting, not only for Banach algebras, but for non-necessarily complete and
non-necessarily associative algebras. Also, the class of morphisms in such a
theory may be different. The development of this approach took a lot of time
and in the present moment is far from the end (only one paper of the planned
series, \cite{rad1}, is published). Thus the proof of the generalized
Berger-Wang formula was unpublished for several years. The aim of the present
publication is to make the proof available for specialists.

It should be noted that the list of generalized Berger-Wang formulae given
here is not complete. The stronger variants of these formulae, as well as an
exposition of some topics of topological radicals, will be published elsewhere.


\begin{thebibliography}{99}                                                                                               %


\bibitem {BW}M. A. Berger, Y. Wang, Bounded semigroups of matrices, Linear
algebra Appl. 166 (1992), 21-27.

\bibitem {Dix}P. G. Dixon, Topologically irreducible representations and
radicals in Banach algebras, Proc. London Math. Soc., (3) 74 (1997), 174-200.

\bibitem {Gu}P. S. Guinand, On quasinilpotent semigroup of operators, Proc.
Amer. Math. Soc. 86 (1982), 485-486.

\bibitem {Lom}V. I. Lomonosov, Invariant subspaces for operators commuting
with compact operators, Funct. Anal. Appl. 7 (1973) 213-214.

\bibitem {PW}J. R. Peters and R. W. Wogen, Commutative radical operator
algebras, J. Operator Theory 42 (1999), 405-424.

\bibitem {RS}G.-C. Rota, W. G. Strang, A note on the joint spectral radius,
Indag. Math. 22 (1960), 379-381.

\bibitem {Sh84}V. S. Shulman, On invariant subspaces of Volterra operators,
Functional Anal. i Prilozen. 18 (2) (1984), 84-85 (Russian).

\bibitem {ShT2000}V. S. Shulman, Yu. V. Turovskii, Joint spectral radius,
operator semigroups and a problem of W.Wojty\'{n}ski, J. Funct. Anal.
177(2000), 383-441.

\bibitem {ShT2001}V. S. Shulman, Yu. V. Turovskii, Radicals in Banach algebras
and some problems of the theory of radical Banach algebras, Functional Anal. i
Prilozen. 35 (2001), no. 4, 88-91.

\bibitem {ShT2002}V. S. Shulman, Yu. V. Turovskii, Formulae for joint spectral
radii of sets of operators, Studia Math. 149 (2002), 23-37.

\bibitem {rad1}V. S. Shulman, Yu. V. Turovskii, Topological radicals, I. Basic
properties, tensor products and joint quasinilpotence, Topological algebras,
their applications and related topics, Banach center publications, volume 67,
pages 293-333, Warszawa 2005.

\bibitem {T85}Yu. V. Turovskii, Spectral properties of certain Lie subalgebras
and the spectral radius of subsets of a Banach algebra, in: Spectral theory of
operators and its applications, \textbf{6 }(1985), "Elm", Baku, 144-181 (in Russian).

\bibitem {Tur98}Yu. V. Turovskii, Volterra semigroups have invariant
subspaces, J.Functional Anal. 162 (2) (1999), 313-323.

\bibitem {WW}W. Wojtynskii, Quasinilpotent Banach-Lie algebras are
Baker-Hausdorff, J. Funct. Anal. 153 (2) (1998), 405-413.
\end{thebibliography}
\end{document}